\documentclass[11pt]{amsart}

\usepackage{amsfonts,latexsym}
\usepackage{amsmath}
\usepackage{amstext}
\usepackage{amssymb}
\usepackage{graphicx}

\newcommand{\R}{\mathbb{R}}

\newcommand{\N}{\mathbb{N}}

\newcommand{\NN}{{\mathcal{N}}}

\newcommand{\EE}{\mathcal{E}}

\newcommand{\eg}{{\it e.g. }}

\newcommand{\lp}{\lambda_p}

\newcommand{\ssum}{\displaystyle\sum}
\newcommand{\llim}{\displaystyle\lim}
\newcommand{\pprod}{\displaystyle\prod}

\newcommand{\mmin}{\displaystyle\min}
\newcommand{\mmax}{\displaystyle\max}

\newcommand{\rr}{{\bf r}}
\newcommand{\nn}{{\lfloor \frac{n}{2} \rfloor}}

\newcommand{\rrr}{{\rr(t_{p+1})}}

\newcommand{\ratio}{\rho}

\newtheorem{theorem}{Theorem}
\newtheorem{lemma}[theorem]{Lemma}
\newtheorem{proposition}[theorem]{Proposition}

\theoremstyle{definition}
\newtheorem{definition}{Definition}
\theoremstyle{remark}

\newtheorem{assumption}{Assumption}

\begin{document}

\title[Consensus under persistent connectivity]{Continuous-Time Consensus under Persistent Connectivity and Slow Divergence of\\ Reciprocal Interaction Weights}


\author[Samuel Martin]{Samuel Martin}
\address{Laboratoire Jean Kuntzmann \\
Universit\'e de Grenoble \\
B.P. 53, 38041 Grenoble, France} \email{samuel.martin@imag.fr}

\author[Antoine Girard]{Antoine Girard}
\address{Laboratoire Jean Kuntzmann \\
Universit\'e de Grenoble \\
B.P. 53, 38041 Grenoble, France} \email{antoine.girard@imag.fr}


\maketitle


\begin{abstract}
In this paper, we present new results on consensus for continuous-time multi-agent systems. 
We introduce the assumptions of persistent connectivity of the interaction graph 
and of slow divergence of reciprocal interaction weights. 
Persistent connectivity can be considered as the counterpart of the notion of ultimate connectivity used in 
discrete-time consensus protocols. 
Slow divergence of reciprocal interaction weights generalizes the assumption of cut-balanced interactions.
We show that under these two assumptions, the continuous-time consensus protocol succeeds: the 
states of all the agents converge asymptotically to a common value. Moreover, our proof allows us to give an estimate of the rate of convergence towards the consensus.
We also provide examples that make us think that both of our assumptions are tight.
\end{abstract}

\section{Introduction}

In multi-agent systems, consensus algorithms serve to emulate the process of agreement: agents exchange information in order to decrease the distance between their states (representing e.g. positions, velocities or opinions depending on the application). 
A multi-agent system is said to reach a consensus when the states of all agents converge asymptotically toward a common value. 
Suitable conditions for convergence to a consensus are typically based on the topology of the network (or graph) of interactions and on the weights of these interactions. 
Consensus algorithms have attracted a lot of attention in the past decade. Notable convergence results include~\cite{Jadbabaie2003,Moreau2005,Ren2005,Blondel2005,Hendrickx2011} for the discrete time and~\cite{Saber2004,Moreau2004,Ren2005,Hendrickx2011,Cao2011,Cao2011dir} for the continuous time consensus algorithm. 

We can classify these results depending on whether or not they require some notion of {reciprocity} in the interaction weights.
Types of reciprocity include:
\begin{itemize}
\item {\it strongly symmetric interactions}: when agent $i$ influences agent $j$  via interaction weight $a_{ji}(t)>0$, agent $j$ also influences agent $i$ via the same weight $a_{ij}(t)=a_{ji}(t)$;

\item {\it weakly symmetric interactions}: when agent $i$ influences agent $j$ via interaction weight $a_{ji}(t)>0$, agent $j$ also influences agent $i$ via a weight $a_{ij}(t)>0$;

\item {\it balanced interactions}: the sum of the interaction weights from and to an agent $i$ are equal, i.e. $\sum_{j\ne i} a_{ji}(t)=\sum_{j\ne i} a_{ij}(t)$;

\item {\it cut-balanced interactions}: for any subgroup  $S$ of agents,
the ratio of reciprocal weights $\sum_{i\in S, j \notin S}a_{ij}(t) \; / \; \sum_{i\in S, j \notin S}a_{ji}(t)$ is bounded.
\end{itemize}

The results presented in this paper also assume some kind of reciprocity in the interaction weights. This allows us to consider weaker assumptions on the connectivity of the interaction graph. We briefly review the most general results in the literature. 
In discrete time, the most general result regarding consensus with reciprocity can be found in~\cite{Moreau2005} or~\cite{Blondel2005}: it is shown that under weakly symmetric interactions, the consensus is achieved whenever the condition of {\it ultimate connectivity} of the interaction graph is satisfied : for all time $t \ge 0$, the union of the interaction graphs over the time interval $[t,+\infty)$ should be connected. In these results, the interaction weights are assumed to be bounded. 
In continuous time, the recent result from Hendrickx and Tsitsiklis~\cite{Hendrickx2011} shows that consensus is achieved under cut-balanced interactions and connectivity of the unbounded interaction graph (the graph with edges $(j,i)$ if $\int_0^\infty a_{ij}(t)dt = +\infty$). 
The same result had been obtained assuming strongly symmetric interactions in~\cite{Cao2011} where the connectivity of the unbounded interaction graph is called infinite integral connectivity.

In this paper, we extend the continuous-time result from~\cite{Hendrickx2011} replacing the cut-balanced interactions assumption by a weaker one: we assume {\it slow divergence of reciprocal weights} (the ratio of reciprocal weights is at worst slowly diverging to infinity). This enables the reciprocal interaction weights to be indefinitely far apart. The proof of our result differs from the one in~\cite{Hendrickx2011} and it allows us to give an explicit bound on the convergence rate to consensus. 
%
The remaining of the paper is organized as follows. 
Section~\ref{sec:pb-statement} formally introduces the consensus algorithm we will be dealing with and states the main results of the paper; in Section~\ref{sec:as}, we discuss the tightness of our assumptions using two examples and provide a comparison with related results. Finally, Section~\ref{sec:conv} presents the proof of our main result.

\section{Problem Statement and Main Results}\label{sec:pb-statement}

The system we study in this paper consists of $n$ {\it agents} interacting  with each other according to a continuous-time consensus protocol. Agents are labeled from $1$ to $n$ and $\NN = \{1,\ldots, n\}$ denotes the label set of the agents. Agents adjust their {\it positions} $x_i(t) \in \R$ for $i\in \NN$ according to the following differential equation
\begin{equation}\label{sys}
\dot{x}_i(t) = \ssum_{j=1}^{n}a_{ij}(t)(x_j(t)-x_i(t)),\; i \in \NN
\end{equation}
where for all $i,j \in \NN$, the {\it interaction weight} $a_{ij}$ represents the strength of the influence of agent $j$ on agent $i$ and is a non-negative measurable function of time, summable on bounded intervals of $\R^+$. 
In its vector form, the equation can be written as $\dot{x}(t) = A(t) x(t)$ where the matrix $A$ is summable on finite intervals of $\R^+$.

Since the right-hand-side of the differential equation~(\ref{sys} may be discontinuous, the equation should be understood as a {\it Carath\'eodory differential equation} (see \eg \cite{Filippov1988}).  A solution to equation (\ref{sys}) is then a locally absolutely continuous function of time $x:\R^+ \rightarrow \R^n$ which satisfies for all $t\in \R^+$ the integral equation:
$$
x_i(t) = x_i(0) + \int_0^t \ssum_{j=1}^{n}a_{ij}(s)(x_j(s)-x_i(s))ds,\; i \in \NN.
$$
\begin{theorem}\label{th:existence-uniqueness}
For any given initial positions $x_i(0)=x_i^0$, $i\in \NN$, the solution to the Carath\'eodory differential equation (\ref{sys}) where the matrix $A$ is summable on finite intervals of $\R^+$ exists and is unique.
\end{theorem}

The proof of this result is a direct consequence of Theorem~54 and Proposition~C.3.8 in~\cite[pages 473-482]{Sontag98}.
We call a solution to equation (\ref{sys}), a {\it trajectory} of the system.
We say that a trajectory {\it reaches a consensus} when $\lim_{t\rightarrow +\infty} x_i(t)$ exist and are the same, for all $i\in \NN$. The common limit is called the {\it consensus value}.
%
We define the {\it group diameter} as  
\begin{equation}\label{eq:diameter}
\Delta_\NN(t) = \mmax_{i\in \NN} x_i(t) - \mmin_{j\in \NN} x_j(t).
\end{equation}
It can be easily shown that $\max_{i\in \NN} x_i(t)$ is non-increasing and that 
$\min_{j\in \NN} x_j(t)$ is non-decreasing. Then, it is clear that the group diameter is non-increasing and that the trajectory reaches a consensus if and only if 
$\llim_{t\rightarrow +\infty} \Delta_\NN(t) = 0.$

Our convergence result involves assumptions on the interaction weights.  
Let $S$ be some non-empty proper subset of $\NN$, we define the ratio between reciprocal interaction weights from $S$ to $\NN \setminus S$:
\begin{equation*}
r_S(t) =  \left\{
\begin{array}{ll}
\frac{\ssum_{i\in S, j \notin S}a_{ij}(t)}{\ssum_{i\in S, j \notin S}a_{ji}(t)}
&\text{if the denominator is positive,} \\
1 &\text{if numerator and denominator are equal to zero,} \\
+\infty &\text{if the denominator is zero and the numerator is positive.}
\end{array}
\right.
\end{equation*}
Then, we define the maximal ratio between reciprocal interaction weights as follows:
\begin{equation}\label{eq:r}
r(t) = \mmax_{ S\neq \emptyset, S\subsetneq \NN} r_S(t)
\end{equation}
For manipulation purposes, we shall use the maximal value of $r$ over all past times. Let 
\begin{equation}\label{eq:rmax}
\rr(t) = \sup_{s\in[0,t]} r(s).
\end{equation}
As defined, $\rr$ is a non-decreasing function of time, and $\rr(t)$ is always greater than $1$. A direct consequence of this definition is the following statement:
\begin{equation}\label{eq:r-bound}
\forall S\neq \emptyset, S\subsetneq \NN, \: \forall s \in [0,t], \; \frac{1}{\rr(t)}\ssum_{i\in S, j \notin S}a_{ij}(s) \leq \ssum_{i\in S, j \notin S}a_{ji}(s) \leq \rr(t) \ssum_{i\in S, j \notin S}a_{ij}(s).
\end{equation}
Thus, whenever a subgroup $S$ of agents influences the rest of the group via interaction weights of sum $a(s)>0$ at time $s\le t$, we know that $S$ is influenced back via interaction weights of sum no less than $\frac{a(s)}{\rr(t)}$.

In our convergence result, we shall make two assumptions on interaction weights. 
The first one is concerned with the topology of the interactions and involves the notion of strong connectivity. A graph, with $\NN$ as set of vertices, is said to be strongly connected when there is a directed path going from node $i$ to node $j$ for all distinct nodes $i, j\in\NN$.
\begin{assumption}[Persistent Connectivity]\label{as:persistent-conn}
The graph $(\NN,\tilde{\EE})$ is strongly connected where 
$$\tilde{\EE} = \left\{(j,i) \in \NN \times \NN \; | \; \int_0^{+\infty}a_{ij}(s)ds = +\infty\right\}.$$
\end{assumption}

This assumption allows us to define a sequence of time instants $(t_p)_{p\in \N}$ which implicitly defines a rescaling of time according to the speed of growth of $\int_0^{t}a_{ij}(s)ds$ for $(j,i)\in \tilde{\EE}$. 
Let $t_0 = 0$ and, for $p\in \N$, let us define $t_{p+1}$ as the last element of the intermediate finite sequence $(t_p^0,t_p^1,\ldots,t_p^\nn)$ where $\lfloor \cdot \rfloor$ is the floor function, $t_p^0 = t_p$ and 
for $q\in \{0,\dots,\nn-1\}$, $t_p^{q+1}$ is the smallest time $t\geq t_p^q$ such that
\begin{equation}\label{eq:rescaling}
\mmin_{
S \subsetneq \NN,
S \neq \emptyset}
\left( \ssum_{i \in S} \ssum_{j\in  \NN \setminus S} \int_{t_p^q}^{t} a_{ij}(s)ds \right) = 1.
\end{equation}
Such a $t$ always exists because $(\NN,\tilde{\EE})$ is strongly connected and therefore for all non empty set $S \subsetneq \NN$, there exists $i\in S$ and $j\in \NN \setminus S$ such that $(j,i)\in \tilde{\EE}$.
Essentially, the sequence $(t_p^0,t_p^1,\ldots,t_p^\nn)$ defines time intervals $[t_p^q,t_p^{q+1}]$
over which
the cumulated influence on any subgroup of agents from the rest of the agents is no less than $1$. 
Let us remark that since we assume that the interaction weights $a_{ij}$ are summable on bounded intervals of $\R^+$, it follows that the sequence $(t_p)_{p\in \N}$ goes to infinity as $p$ goes to $+\infty$.

The interactions over intervals $[t_p^q,t_p^{q+1}]$ induce a chain of movements of the agents toward the center of the group. These movements propagate toward agents having either smallest or largest position in less than $\nn$ such intervals and result in a contraction of the group diameter  between $t_p$ and $t_{p+1}$.
We will prove in Section~\ref{sec:conv} the following proposition which constitutes the main contribution of the paper and is the core of our main result:

\begin{proposition}[Group diameter contraction rate]\label{prop:small-move} If Assumption~\ref{as:persistent-conn} (persistent connectivity) holds, then for all $p \in \N$, 
\begin{equation*}
\Delta_\NN(t_{p+1}) \leq \left(1 - \frac{\rr(t_{p+1})^{-\nn}}{(8n^2)^\nn} \right)
\Delta_\NN(t_p).
\end{equation*}
\end{proposition}

It is clear from the previous proposition that the sequence $({\rr(t_{p})})_{p\in \N}$ plays a central role in the fact that the consensus is reached or not. This is where our second assumption regarding the interaction weights comes into play.
\begin{assumption}[Slow divergence of reciprocal interaction weights]\label{as:weak-sym}
For all $t\geq 0$, $\rr(t)$ is finite and the infinite sum
$
\sum_{p\in\N} \rr(t_p)^{-\nn}=+\infty.
$
\end{assumption}

The assumption requires $\rr(t)$ not to grow too fast. For instance, $\rr(t_p) = O(p^{2/n})$ (which includes the case where $\rr$ is bounded) satisfies Assumption~\ref{as:weak-sym}, whereas $\rr(t_p) = p^{4/n}$ does not. Hence, the assumption enables the divergence of reciprocal interaction weights provided this divergence is slow. Let us remark that the larger the number of agents, the slower the divergence can be.
We can now state the main result of the paper.
\begin{theorem}\label{th:sym-cont}
If Assumptions~\ref{as:persistent-conn} (persistent connectivity) and~\ref{as:weak-sym} (slow divergence of reciprocal interaction weights) hold, then the trajectory of system~(\ref{sys}) reaches a consensus.
\end{theorem}

\begin{proof}
From Proposition~\ref{prop:small-move} and using a simple induction we can show that for all $P\ge 1$,
\begin{equation}
\label{eq:diam}
\Delta_\NN(t_{P}) \leq \prod_{p=1}^P\left(1 -\frac{\rr(t_{p+1})^{-\nn}}{(8n^2)^\nn} \right)
\Delta_\NN(0).
\end{equation}
Then, the product in the right-hand side of the inequality converges to $0$ when $P$ goes to infinity if and only if $\sum_{p\in\N} \rr(t_p)^{-\nn}=+\infty$.
This last statement is true according to Assumption~\ref{as:weak-sym}. Consequently, the group diameter goes to zero and the consensus is achieved.
\end{proof}

We would like to point out that not only Proposition~\ref{prop:small-move} allows us to prove that the consensus is reached, but it also provides an estimate of the convergence rate to the consensus value.
In the following section, our  assumptions are discussed in more details. 

\section{Discussion on the Assumptions}\label{sec:as}

\subsection{Persistent connectivity}\label{sec:as1}
First, let us show why Assumption~\ref{as:persistent-conn} is more suitable for the continuous-time consensus protocol~(\ref{sys}) than the notion of ultimate connectivity, which is often used for the discrete-time protocols in~\cite{Moreau2005,Blondel2005}. Formally, ultimate connectivity is defined as follows:
\begin{assumption}[Ultimate connectivity]\label{as:ultimate-conn}
For all $t\ge 0$, the graph $G(t) = (\NN,\bigcup_{s\geq t} \EE(s))$ is strongly connected where $\EE(s) \subseteq \NN\times\NN$ is the set of directed interaction links such that $(j,i) \in \EE(s)$ if and only if $a_{ij}(s)>0$.
\end{assumption}

Using a counterexample, we show that a continuous-time system which respects the ultimate connectivity assumption and bears bounded and symmetric interaction weights does not always converge to consensus. This implies that the result established in~\cite{Moreau2005,Blondel2005} for discrete-time consensus protocols cannot be directly transposed to the continuous-time case.
Consider a system of form~(\ref{sys}) of two agents with $x_1(0)< x_2(0) $ and with interaction weights 
\begin{equation*}
\forall t \in \R^+,\; a_{12}(t) = a_{21}(t) = \left\{
\begin{array}{ll}
1 \text{ if } \exists k \in \N, t \in [k,k+\frac{1}{2^{k+1}}] \\
0 \text{ otherwise}
\end{array}
\right.
\end{equation*}
The weights are bounded and symmetric. Under these conditions, Theorem~\ref{th:existence-uniqueness} provides existence and uniqueness of a solution to the above differential equation understood in the sense of Carath\'eodory differential equation. Also, it is clear that Assumption~\ref{as:ultimate-conn} holds.
Now, solving the differential equation~(\ref{sys}), we obtain for all $k\in \N$, 
$$
x_2(k+1)-x_1(k+1)=e^{-2 \int_k^{k+1} a_{12}(s)ds
}(x_2(k)-x_1(k)) =e^{-\frac{1}{2^k}} (x_2(k)-x_1(k)),
$$
which yields for all $k\in \N$,
$$
x_2(k)-x_1(k) =  \pprod_{h=0}^{k-1} e^{-\frac{1}{2^h}} (x_2(0)-x_1(0))
 = e^{-\frac{1-1/2^k}{1-1/2}}(x_2(0)-x_1(0)) \ge e^{-2}(x_2(0)-x_1(0)).
$$
This proves that the system does not reach a consensus. Notice that this could be alleviated by forcing $\int_0^t a_{12}(s)ds$ to diverge toward $+\infty$ when $t$ tends to $+\infty$ which would correspond to Assumption~\ref{as:persistent-conn}. Hence, this example explains why persistent connectivity is more relevant than ultimate connectivity for continuous-time consensus.

\subsection{Slow divergence of reciprocal interaction weights}\label{sec:as2}
\label{sec:tight-example}

We present a simple example for which Assumption~\ref{as:weak-sym} is tight.
Let $(\ratio_p)_{p\in \N}$ be a non-decreasing sequence such that $\ratio_p\ge 1$, for all $p\in \N$.
Let us consider a multi-agent system with $3$ agents where $x_1(0) < x_2(0) <x_3(0)$ and with the dynamics given by
$$
\left\{
\begin{array}{l}
\dot{x}_1(t) = x_2(t) - x_1(t) \\
\dot{x}_2(t) = \ratio_p(x_1(t) - x_2(t)) \\
\dot{x}_3(t) = 0 
\end{array}
\right., \; \text{ if } t\in [2p,2p+1)
$$
and
$$
\left\{
\begin{array}{l}
\dot{x}_1(t) = 0 \\
\dot{x}_2(t) = \ratio_p(x_3(t) - x_2(t)) \\
\dot{x}_3(t) = x_2(t)-x_3(t)
\end{array}
\right., \; \text{ if } t\in [2p+1,2p+2).
$$


This system satisfies the persistent connectivity Assumption~\ref{as:persistent-conn} (with $(\NN,\tilde{\EE})$ being the undirected line graph). 
The sequence $(t_p)_{p\in \N}$ acting as a rescaling of time and defined by equation (\ref{eq:rescaling}) is given  for $p \in \N$ by,
$t_p = 2p$, and $\rr(t_{p})=\ratio_{p}$.
Then, Assumption~\ref{as:weak-sym} holds if and only if
$\sum_{p\in\N} {\ratio_{p}}^{-1}=+\infty$.
The validity of Assumption~\ref{as:weak-sym} hence depends on the choice of sequence $(\ratio_p)_{p\in \N}$. In this specific case, we can show that Assumption~\ref{as:weak-sym} is necessary and sufficient for the trajectory of the system to reach a consensus. We know from Theorem~\ref{th:sym-cont} that it is a sufficient condition. 
Then, let us show that if $\sum_{p\in\N} {\ratio_{p}}^{-1}<+\infty$ then the consensus cannot be reached.

It is easy to see that $x_1(t)$ is non-decreasing, $x_3(t)$ is non-increasing and for all $t\ge 0$, $x_1(t)\le x_2(t) \le x_3(t)$. Integrating the dynamics of the system, we can show that for all $p\in \N$:
$$
\left\{
\begin{array}{lll}
x_1(2p+1) &=& (1-\lambda_p) x_1(2p) + \lambda_p x_2(2p) \\
x_2(2p+1) &=& \ratio_p\lambda_p x_1(2p) + (1-\ratio_p\lambda_p) x_2(2p) \\
x_3(2p+1) &=& x_3(2p)
\end{array}
\right.
$$
and
$$
\left\{
\begin{array}{lll}
x_1(2p+2) &=& x_1(2p+1) \\
x_2(2p+2) &=& (1-\ratio_p\lambda_p) x_2(2p+1) + \ratio_p\lambda_p x_3(2p+1) \\
x_3(2p+2) &=& \lambda_p x_2(2p+1) + (1-\lambda_p) x_3(2p+1)
\end{array}
\right.
$$
where $\lp = \frac{1-e^{-(\ratio_p+1)}}{1+\ratio_p}$.
Then, let us remark that
\begin{eqnarray*}
x_3(2p+2) &=& \lambda_p x_2(2p+1) + (1-\lambda_p) x_3(2p+1)\\
& \ge & \lambda_p x_1(2p+1) + (1-\lambda_p) x_3(2p+1) \ge
 \lambda_p x_1(2p) + (1-\lambda_p) x_3(2p). 
\end{eqnarray*}
Also,
$$
x_1(2p+2) = x_1(2p+1) = (1-\lambda_p) x_1(2p) + \lambda_p x_2(2p) \le
(1-\lambda_p) x_1(2p) + \lambda_p x_3(2p). 
$$
Therefore,
$$
x_3(2p+2)-x_1(2p+2) \geq (1-2\lp)(x_3(2p)-x_1(2p)).
$$
By induction on the previous equation, we obtain for all $P\in \N$,
\begin{equation}
\label{eq:min}
x_3(2P)-x_1(2P) \ge \pprod_{p=0}^{P-1} (1-2\lambda_p)  (x_3(0)-x_1(0)).
\end{equation}
If $\sum_{p\in\N} {\ratio_{p}}^{-1}<+\infty$, then necessarily the limit of $\ratio_p$ is infinite.
Then, $\lambda_p \le {\ratio_p}^{-1}$, and therefore $\sum_{p \in \N} \lambda_p < +\infty$.
This implies that the left-hand side of the inequality (\ref{eq:min}) does not tend to zero and the trajectory of the system cannot not reach a consensus.
Therefore, for that particular system, Assumption~\ref{as:weak-sym} is a necessary and sufficient condition for consensus. 

We also realized a numerical study of an extension of the previous example. Let us consider a system with $n=2m+1$ agents with $m \ge 2$ whose dynamics is defined as follows:
\begin{itemize}
\item For $t\in [(m+1)p+i,(m+1)p+i+1)$ with $p\in \N$, $i\in \{0,\dots,m-2\}$,
$$
\left\{
\begin{array}{llll}
\dot{x}_{i+1}(t)   &= & x_{i+2}(t) - x_{i+1}(t) \\
\dot{x}_{i+2}(t)   &= & \ratio_p(x_{i+1}(t) - x_{i+2}(t)) \\
\dot{x}_{n-1-i}(t) &= &\ratio_p(x_{n-i}(t) - x_{n-1-i}(t)) \\
\dot{x}_{n-i}(t) &= &x_{n-1-i}(t) - x_{n-i}(t) \\
\dot{x}_j(t) &=& 0 & \text{ if } j\notin\{i+1,i+2,n-1-i,n-i\}
\end{array}
\right.
$$
\item For $t\in [(m+1)p+m-1,(m+1)p+m)$ with $p\in \N$,
$$
\left\{
\begin{array}{llll}
\dot{x}_{m}(t)   &= & x_{m+1}(t) - x_{m}(t) \\
\dot{x}_{m+1}(t)   &= & \ratio_p(x_{m}(t) - x_{m+1}(t)) \\
\dot{x}_j(t) &=& 0 & \text{ if } j\notin\{m,m+1\}
\end{array}
\right.
$$
\item For $t\in [(m+1)p+m,(m+1)(p+1))$ with $p\in \N$,
$$
\left\{
\begin{array}{llll}
\dot{x}_{m+1}(t)   &= & \ratio_p(x_{m+2}(t) - x_{m+1}(t)) \\
\dot{x}_{m+2}(t)   &= & (x_{m+1}(t) - x_{m+2}(t)) \\
\dot{x}_j(t) &=& 0 & \text{ if } j\notin\{m+1,m+2\}
\end{array}
\right.
$$
\end{itemize}
Regarding the initial conditions, we set $x_j(0) = -1$ for $i \in \{0,\ldots,m-1\}$, $x_m(0) = 0$ and $x_i(0) = 1$ for $i \in \{m+1,\ldots,2m\}$.
This system satisfies the persistent connectivity Assumption~\ref{as:persistent-conn}. 
We can also show that we have $\rr(t_{p+1})=\ratio_{p}$, for $p\in \N$.
Then, Assumption~\ref{as:weak-sym} holds if and only if
$\sum_{p\in\N} {\ratio_{p}}^{-m}=+\infty$. In the following, we show the results of our numerical simulations of the system with $11$ agents (i.e. $m=5$).
We simulated the system for three different sequences $(\ratio_p)_{p\in \N}$: $\ratio_p=1$, $\ratio_p=(1+p)^{\frac{1}{5}}$, $\ratio_p=(1+p)^{\frac{2}{5}}$. 
It should be noted that Assumption~\ref{as:weak-sym} holds for the first two sequences but not for the third one. The results of the simulations are shown in Figure~\ref{fig} where we represented the evolution of the diameter $\Delta_\NN(t)$ over time. The simulations are consistent with the theory showing that the consensus is reached for the first two sequences (the diameter goes to zero). Also, for the third sequence, we can observe that the consensus is not reached. For the first two sequences, we also represented the evolution of the diameter $\Delta_\NN(t)$ in a logarithmic scale in order to estimate the convergence rate. It appears clearly that, for the first sequence, the convergence rate is exponential. For the second sequence, the convexity of the curve indicates that the convergence rate is sub-exponential.

\begin{figure}[!t]
\begin{center}
\includegraphics[width=0.65\columnwidth]{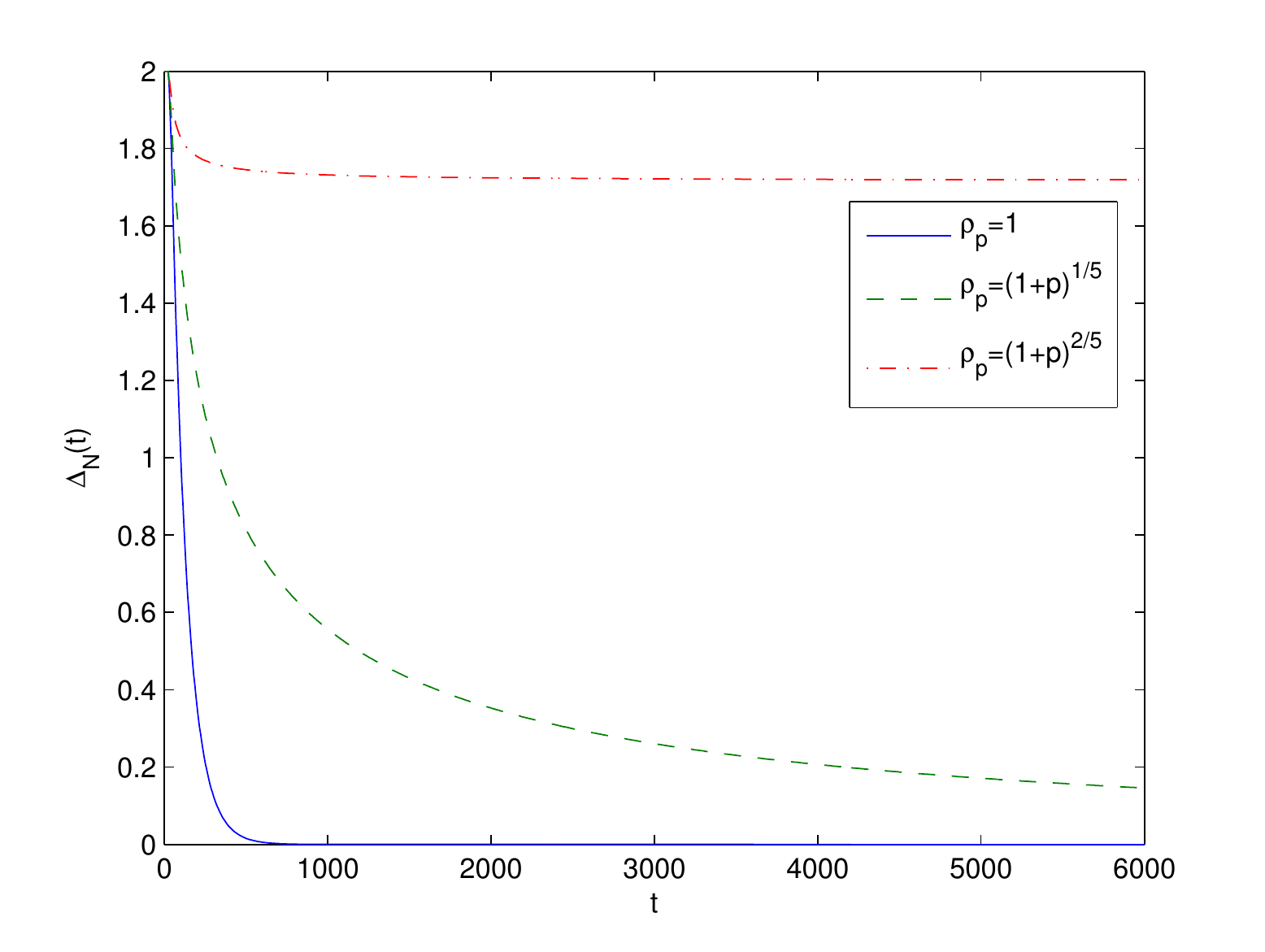}  \\
\includegraphics[width=0.48\columnwidth]{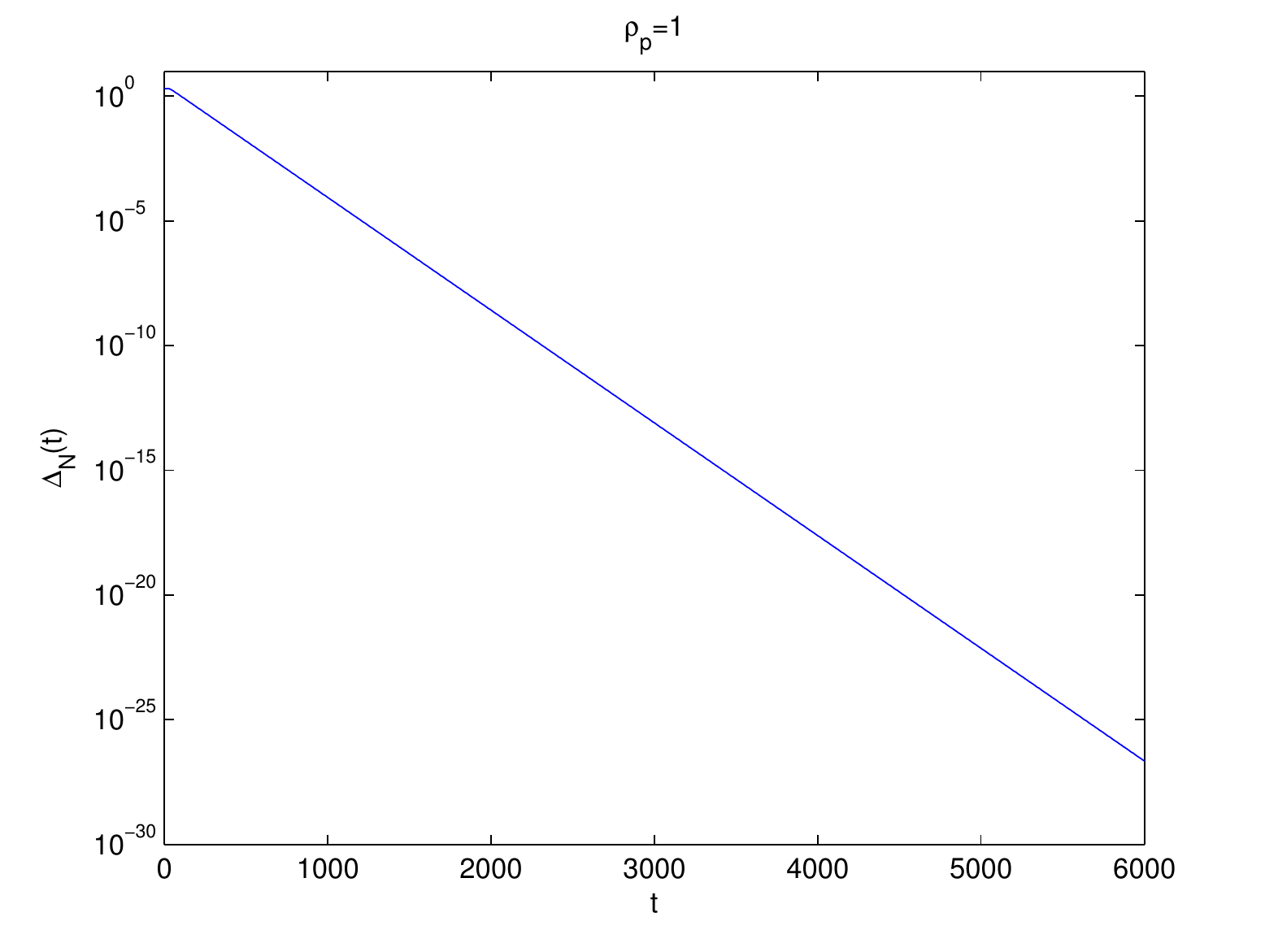}  
\includegraphics[width=0.48\columnwidth]{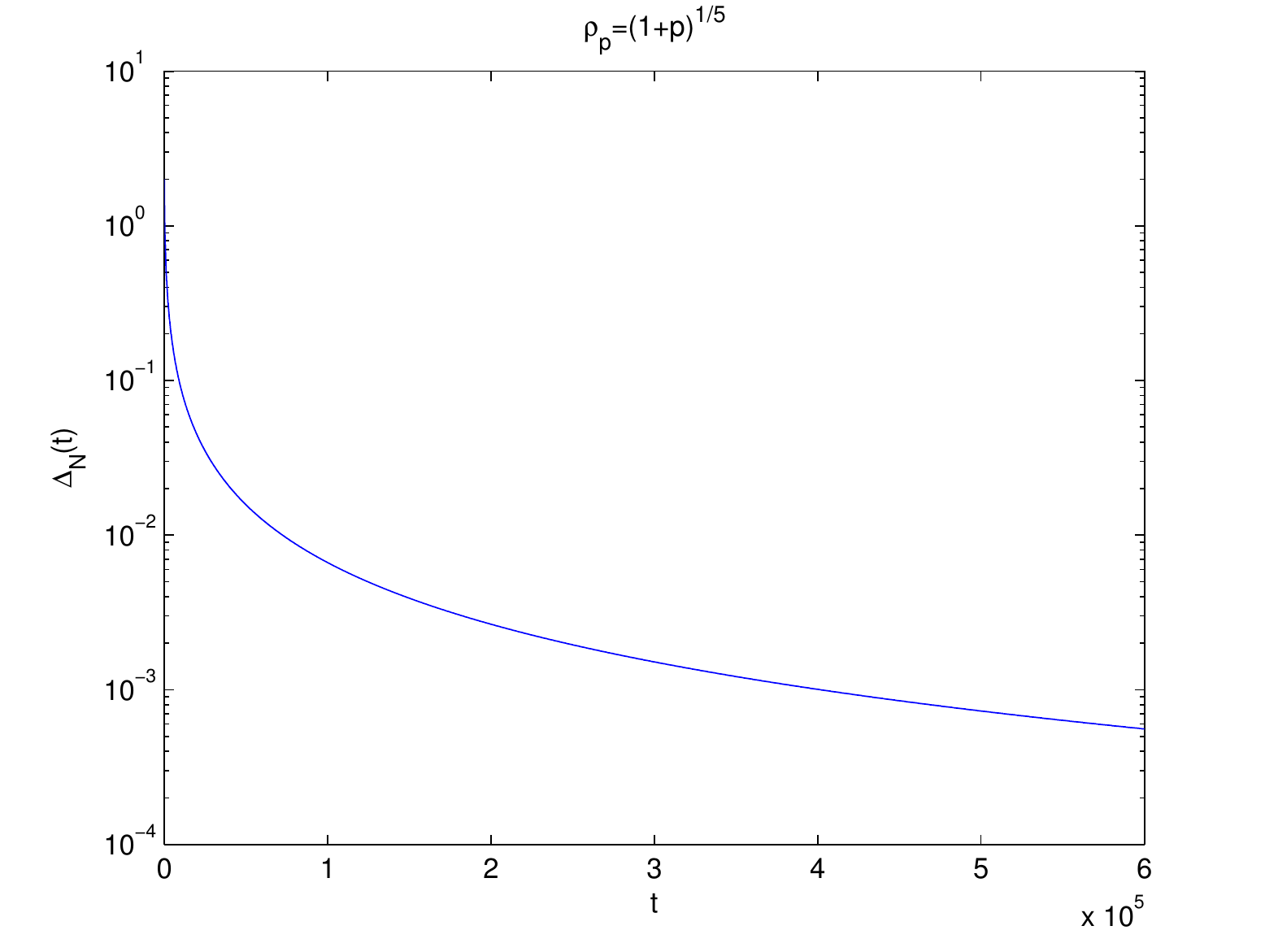}  
\caption{Top: evolution of the diameter $\Delta_\NN(t)$ for the system with $11$ agents for the sequences given by $\ratio_p=1$ (plain line), $\ratio_p=(1+p)^{\frac{1}{5}}$ (dashed line), $\ratio_p=(1+p)^{\frac{2}{5}}$ (dashes and dots). Bottom: evolution of the diameter $\Delta_\NN(t)$ in logarithmic scale for the sequences given by $\ratio_p=1$ (left), $\ratio_p=(1+p)^{\frac{1}{5}}$ (right).
}  
\label{fig}                                 
\end{center}                                 
\end{figure}

\subsection{Related results}\label{sec:related-results}
Hendrickx and Tsitsiklis proved recently~\cite{Hendrickx2011} that consensus is achieved provided that Assumption~\ref{as:persistent-conn} is satisfied along with the following assumption:
\begin{assumption}[Cut-balanced weights]\label{as:h}
There exists $K\geq 1$ such that for all non-empty strict subset $S \subseteq \NN$ and for all time $t\ge 0$, 
$$K^{-1}  \ssum_{i \in S, j\notin S} a_{ji}(t) \leq \ssum_{i \in S, j\notin S} a_{ij}(t) \leq K \ssum_{i \in S, j\notin S} a_{ji}(t).$$
\end{assumption}
Assumption~\ref{as:h} generalizes both symmetric weights ($a_{ij}(t) = a_{ji}(t)$ for $i,j \in \NN$) and balanced weights ($\sum_{j\ne i} a_{ij}(t) = \sum_{j\ne i} a_{ji}(t)$ for $i \in \NN$). In particular, the result from~\cite{Hendrickx2011} includes those from~\cite{Cao2011} (using strongly symmetric weights) and~\cite{Saber2004} (using balanced weights).

Let us assume that Assumptions~\ref{as:persistent-conn} and~\ref{as:h} are satisfied. Then, according to the definition of the maximal ratio between reciprocal interaction weights given by~(\ref{eq:r}) and~(\ref{eq:rmax}), $\rr(t) \le K$ for all $t\in \R^+$ which gives that $\rr(t)$ is bounded, thus Assumption~\ref{as:weak-sym} is satisfied. 
Therefore, Theorem~\ref{th:sym-cont} states that trajectories of system~(\ref{sys}) reach a consensus.
Hence, the convergence result in~\cite{Hendrickx2011}
can be regarded as a particular case of Theorem~\ref{th:sym-cont}. 

Let us remark that Theorem~\ref{th:sym-cont} is more general though. Indeed, looking at the example presented in section~\ref{sec:tight-example}, it is clear that Assumption~\ref{as:h} may not be satisfied even though consensus is reached (i.e. if $\ratio_p=p$, for all $p\in \N$). In that case, the result in~\cite{Hendrickx2011} does not allow us to conclude that consensus is reached while Theorem~\ref{th:sym-cont} does.

Also, there is no estimation of the rate of convergence to the consensus in~\cite{Hendrickx2011} whereas Proposition~\ref{prop:small-move}
allows us to state the following result:
\begin{proposition} Let Assumptions~\ref{as:persistent-conn} (persistent connectivity) and~\ref{as:h} (cut-balanced weights) hold. Let $P : \R^+ \rightarrow \N$ be the function defined for all $t\in \R^+$ by
$P(t)=p$ if $t\in [t_p,t_{p+1})$. Then,
$$
\forall t\in \R^+,\;
\Delta_\NN(t) \leq \left(1 -\frac{1}{(8Kn^2)^\nn} \right)^{P(t)}
\Delta_\NN(0).
$$
\end{proposition}

\begin{proof} Using the fact that  for all $t\in \R^+$ $\rr(t) \le K$, equation (\ref{eq:diam}) 
gives for all $P \ge 1$,
$$
\Delta_\NN(t_{P}) \leq \prod_{p=1}^{P} \left(1 -\frac{1}{(8Kn^2)^\nn} \right)
\Delta_\NN(0)= \left(1 -\frac{1}{(8Kn^2)^\nn} \right)^{P}
\Delta_\NN(0).
$$
Then, since $\Delta_\NN$ is non-increasing, we have for all $t\in \R^+$, $\Delta_\NN (t) \le \Delta_\NN(t_{P(t)})$
which allows us to conclude.
\end{proof}

In particular, the previous proposition states that if $P(t)$ grows linearly,
then the consensus is approached at an exponential rate. On the other hand, if $P(t)$ grows logarithmically, 
then the consensus is approached only at polynomial rate.
This can be illustrated using the system with $2$ agents with symmetric interaction weights $a_{12}(t) = a_{21}(t) = 1/t$.
For that system, Assumptions~\ref{as:persistent-conn}  and~\ref{as:h}  hold. We can show that, for $p\in \N$, $t_p=e^p$ which gives
$P(t)=\lfloor \ln(t) \rfloor$. Hence, the previous proposition states that the consensus is reached at polynomial rate which is indeed the case
since the direct resolution of the system gives $\Delta_\NN(t)=\Delta_\NN(0)/{t}$.

Another related result was established by Moreau in~\cite{Moreau2004} where it is shown that under the following assumption, trajectories of system~(\ref{sys}) reach a consensus :
\begin{assumption}[Connectivity of integral $\delta$-digraph]\label{as:delta-conn}
The weights $a_{ij}$ are bounded and piecewise continuous. There exists $\delta>0$ and interval length $T>0$ such that for all $t>0$, the graph $(\NN,\EE_{\delta,T}(t))$ has a spanning tree where 
$$\EE_{\delta,T}(t) = \left\{(j,i) \in \NN \times \NN \; | \; \int_t^{t+T}a_{ij}(s)ds > \delta \right\}.$$
\end{assumption}
Unlike our assumption, the previous one does not require strongly connected interaction graphs. However, the connectivity is required on every interval of constant length $T$. 
To illustrate the difference, we consider again the system with $2$ agents with symmetric interaction weights $a_{12}(t) = a_{21}(t) = 1/t$. Let $\delta>0$, $T>0$ and $t>0$,
we have
$$\int_t^{t+T}a_{12}(s)ds > \delta \iff \ln\left(1+\frac{T}{t}\right) > \delta  \iff t < \frac{T}{e^\delta -1}.$$
Thus, the graph $(\NN,\EE_{\delta,T}(t))$ will be disconnected for times $t$ greater than $\frac{T}{e^\delta -1}$. This invalidates Assumption~\ref{as:delta-conn}. On the other hand, 
Theorem~\ref{th:sym-cont} allows us to conclude that the trajectories of the system reach a consensus whereas this conclusion could not be obtained from \cite{Moreau2004}.

\section{Estimation of the Group Diameter Contraction Rate}\label{sec:conv}

In this section, we give the proof of Proposition~\ref{prop:small-move}.
It is not directly proved for system (\ref{sys}) but for an equivalent system which preserves the relative order of the agents positions.

\subsection{Equivalent order-preserving system}\label{sec:order}
The manipulation of the trajectory $x$ of system (\ref{sys}) would be much easier if the positions of agents always remained in the same order. Unfortunately, this is generally not the case. However, we can imagine an equivalent trajectory $y$ where whenever the order of the positions $x_i$ and $x_j$ of two agents changes, the agents exchange their labels so that the order actually remains unchanged in $y$. This idea of label reordering was also used in~\cite{Hendrickx2011} where the authors managed to show several interesting properties of the sorted trajectory.
We first define the reordering permutation.
\begin{definition}
\label{def:sigma}
For $t\in \R^+$, let $\sigma_t$ be the permutation of $\NN$ verifying for all $i,j \in \NN$,
\begin{center}
$i<j \:\: \Rightarrow  \:\: x_{\sigma_t(i)}(t) < x_{\sigma_t(j)}(t)$  or  $(x_{\sigma_t(i)}(t) = x_{\sigma_t(j)}(t)$ and $\sigma_t(i) < \sigma_t(j) )$.
\end{center}
\end{definition}

This permutation sorts the sequence $(x_i(t),i)$ in lexicographic order.
This leads to defining an order-preserving trajectory $y$ associated to $x$ as follows :
\begin{equation}\label{eq:y}
\forall t \in \R^+, \forall i \in \NN,\: y_i(t) = x_{\sigma_t(i)}(t).
\end{equation}
Several properties of $y$ are straightforward consequences of this definition:
\begin{proposition}\label{prop:order}
Let $y$ be given by (\ref{eq:y}), then 
\begin{itemize}
\item For all $t\in \R^+$, for all $i,j \in \NN$, $i<j \implies y_i(t) \leq y_j(t)$,
\item The functions $y_1$ and $y_n$ are respectively non-decreasing and non-increasing,
\item For all $t\in \R^+$, $\Delta_\NN(t)= y_n(t)-y_1(t)$.
\end{itemize}
\end{proposition}

\begin{proof} The first property is direct consequence of (\ref{eq:y}) and Definition~\ref{def:sigma}.
This property and the fact that $\sigma_t$ is a bijection then gives:
$$
y_1(t)=\min_{j\in \NN} y_j(t)=\min_{j\in \NN} x_{j}(t) \text{ and } y_n(t)=\max_{i\in \NN} y_i(t)= \max_{i\in \NN} x_{i}(t).
$$
The fact that $\min_{j\in \NN} x_{j}(t)$ and $\max_{i\in \NN} x_{i}(t)$ are respectively non-decreasing and non-increasing then gives the second property. Finally, we also have 
$$
\Delta_\NN(t)= \max_{i\in \NN} x_{i}(t)- \min_{j\in \NN} x_{j}(t) =y_n(t)-y_1(t).
$$
\end{proof}


We also define the interaction weights $b_{ij}$ for the order preserving trajectory $y$, given at time $t\in \R^+$ by
$b_{ij}(t) = a_{\sigma_t(i)\sigma_t(j)}(t)$.
The following result, established in~\cite{Hendrickx2011}, 
shows that $y$ is solution of a Carath\'eodory differential equation 
(even in the presence of an infinite number of position exchanges in finite time in the original trajectory $x$) and 
justifies the appellation of equivalent order-preserving system.
\begin{proposition}[{\cite{Hendrickx2011}}]\label{prop:equiv}
For all $t \in \R^+$, $y$ satisfies the equation
\begin{equation}\label{sys:equiv}
y_i(t) = y_i(0) + \int_0^t \ssum_{j\in \NN} b_{ij}(s) (y_j(s) - y_i(s)) ds,\; i\in \NN.
\end{equation}
\end{proposition}

The framework to use our equivalent system is now settled. We now give several properties of this system that will be useful for the proof of Proposition~\ref{prop:small-move}.

\begin{lemma}\label{l:r-bound-b}
The interaction weights $b_{ij}$ satisfy the following inequalities
\begin{equation*}
\forall S\neq \emptyset, S\subsetneq \NN, \: \forall s \in [0,t], \; \frac{1}{\rr(t)}\ssum_{i\in S, j \notin S}b_{ij}(s) \leq \ssum_{i\in S, j \notin S}b_{ji}(s) \leq \rr(t) \ssum_{i\in S, j \notin S}b_{ij}(s).
\end{equation*}
\end{lemma}

\begin{proof}
Let $s$ in $[0,t]$ and $S$ be some non-empty strict subset of $\NN$. Denote $S' = \sigma_s(S)$. Since $\sigma_s$ is a bijection, we have $\sigma_s(\NN \setminus S) = \NN \setminus S'$. Thus, using the definition of  $b_{ij}(s)$,
$$
\ssum_{i\in S, j \notin S}b_{ij}(s) = \ssum_{i\in S, j \notin S}a_{\sigma_s(i)\sigma_s(j)}(s)
= \ssum_{i\in S', j \notin S'} a_{ij}(s)
$$
and
$$
\ssum_{i\in S, j \notin S}b_{ji}(s) = \ssum_{i\in S, j \notin S}a_{\sigma_s(j)\sigma_s(i)}(s)
= \ssum_{i\in S', j \notin S'} a_{ji}(s).
$$
Equation (\ref{eq:r-bound}) with $S:=S'$ allows us to conclude.
\end{proof}

\begin{lemma}\label{l:subbij}
Let $t$ and $t'$ such that $0\le t<t'$, and let us assume that there exist $l \in \{1,\dots,n-1\}$, $c\in \R$ and $C\in \R$ with $c<C$ and satisfying
\begin{equation}\label{eq:minor}
\forall s \in [t,t'],\;  y_l(s) \leq c \text{ and } C \leq y_{l+1}(s).
\end{equation}
Then, for all $s \in [t,t']$, sets $\sigma_s(\{1,\dots,l\})$ and $\sigma_s(\{l+1,\dots,n\})$ remain constant.
\end{lemma}

\begin{proof} Let us assume that $\sigma_\cdot(\{1,\dots,l\})$ does not remain constant. Then, there exists $k \in \sigma_t(\{1,\dots,l\})$ and $s$ in $[t,t']$ such that $k \notin \sigma_s(\{1,\dots,l\})$. Since $\sigma_s$ is a bijection, we have $k$ in $\sigma_s(\{l+1,\dots,n\})$. Denote $i\in\{1,\dots,l\}$ and $j\in \{l+1,\dots,n\}$ such that $k = \sigma_t(i)$ and $k = \sigma_s(j)$. We have
$$
x_k(t) = x_{\sigma_t(i)}(t) = y_i(t)\le y_l(t) \leq c
\text{ and }
x_k(s) = x_{\sigma_s(j)}(s) = y_j(s)\ge y_{l+1}(s) \geq C.
$$
By continuity of $x_k$, there exists $\tau\in (t,s)$ such that $x_k(\tau) = \frac{c+C}{2}$. Denote $m \in \NN$ such that $k = \sigma_\tau(m)$. Then $y_m(\tau) = x_{\sigma_\tau(m)}(\tau)=x_k(\tau) = \frac{c+C}{2}$. Then,
(\ref{eq:minor}) gives $y_l(\tau)<y_m(\tau)<y_{l+1}(\tau)$ which is in contradiction with the fact that the coordinates of $y(\tau)$ are ordered.
\end{proof}

The next lemma comes from an adaptation of a proof from \cite{Hendrickx2011}. There, the authors exhibit Lyapunov-like functions $
S_l(t) = \sum_{i=1}^l R^{-i} y_i(t)$ with positive derivatives~\cite[Lemma $1$]{Hendrickx2011}.
Here, we establish a lower bound on their derivatives:
\begin{lemma}\label{l:adaptation_hendrickx}
Assume that for $i,j\in \NN$, $b_{ij}$ are non-negative weights, $R\geq 1$ is some constant satisfying
\begin{equation*}
\forall S\neq \emptyset, S\subsetneq \NN, \; \frac{1}{R}\ssum_{i\in S, j \notin S}b_{ij} \leq \ssum_{i\in S, j \notin S}b_{ji} \leq R \ssum_{i\in S, j \notin S}b_{ij}.
\end{equation*}
Then, for every sorted vector $y$ in $\R^n$, for all $l\in \NN$
\begin{equation*}
\ssum_{i=1}^l R^{-i} \left( \ssum_{j=1}^n b_{ij} (y_j - y_i) \right) \geq 0
\end{equation*}
and for all $l\in \{1,\dots,n-1\}$,
\begin{equation*}
\ssum_{i=1}^l R^{-i} \left( \ssum_{j=1}^n b_{ij} (y_j - y_i) \right) \geq (y_{l+1}-y_l) R^{-l} \ssum_{i=l+1}^n \ssum_{j=1}^{l} b_{ji}.
\end{equation*} 
\end{lemma}

\begin{proof} The first inequality is established in~\cite{Hendrickx2011}.
The proof of the second inequality is also adapted from~\cite{Hendrickx2011}. Following the cited paper, we use notation $w_k = R^{-k}$ for $k\in \{1\ldots l\}$ and $w_k=0$ for $k>l$. The following inequality is established in~\cite{Hendrickx2011}: 
\begin{equation*}
\ssum_{i=1}^n w_i \left( \ssum_{j=1}^n b_{ij} (y_j - y_i) \right) \geq \ssum_{k=1}^{n-1} (y_{k+1}-y_k)  \left(  w_k \ssum_{i=k+1}^n \ssum_{j=1}^k b_{ji}  - w_{k+1} \ssum_{i=k+1}^n \ssum_{j=1}^k b_{ij}   \right).
\end{equation*}
Using the definition of $w_k$, this inequality can be rewritten as
\begin{eqnarray*}
\ssum_{i=1}^l R^{-i} \left( \ssum_{j=1}^n b_{ij} (y_j - y_i) \right) 
&\geq& \ssum_{k=1}^{l-1} (y_{k+1}-y_k)R^{-k} \left(  \ssum_{i=k+1}^n \ssum_{j=1}^k b_{ji}  - \frac{1}{R} \ssum_{i=k+1}^n \ssum_{j=1}^k b_{ij}   \right) \\
&&+  (y_{l+1}-y_l) R^{-l} \ssum_{i=l+1}^n \ssum_{j=1}^l b_{ji}.
\end{eqnarray*}
The first term in the right-hand side of the inequality is positive by assumptions on the weights $b_{ij}$ and vector $y$. This observation leads to the desired inequality.
\end{proof}

\subsection{Proof of Proposition~\ref{prop:small-move}}

Proposition~\ref{prop:small-move} is the core of our main result. 
Before going into the details, let us give a brief overview of the proof.
If $\Delta_\NN(t_p) >0$, then there exists a two consecutive positions $y_l$ and $y_{l+1}$ whose distance
 at time $t_p$  is greater than $\Delta_\NN(t_p)/n >0$.
We show, using Assumption~\ref{as:persistent-conn} (persistent connectivity) that this {\it gap}~\footnote[1]{We use the term {\it gap} between agents $y_l$ and $y_{l+1}$ since by definition, no other agent falls between $y_l$ and $y_{l+1}$ (see Proposition~\ref{prop:order}).} between $y_l$ and $y_{l+1}$ will result in 
a quantifiable movement of an agent (Lemma~\ref{l:relative-increase}). Then, we show that the movement of this agent creates
a new quantifiable {\it gap} between two other agents (Lemma~\ref{l:diameter-ind}). We show that this process propagates to one of the extremities 
of the group so that there is a contraction of the group diameter on the time interval $[t_p,t_{p+1}]$ (main proof).

\begin{lemma}\label{l:relative-increase} 
Let $q \in \{0,\ldots,\nn-1\}$, $\tau \in [t_p,t_p^q]$ and $l \in \{1,\ldots,n-1\}$ such that
$y_{l}(\tau)<y_{l+1}(\tau)$.
Then, there exists $m \in \{1,\ldots,l\}$  and $\tau'\in [t_p,t_p^{q+1}]$ such that
\begin{equation*}
y_{m}(\tau') - y_{m}(\tau) \geq \frac{\rrr^{m-(l+1)}}{4n}\left( y_{l+1}(\tau)-y_{l}(\tau) \right),
\end{equation*}
\end{lemma}

\begin{proof}
Let us remark that one of the following statements is always satisfied:
\begin{enumerate}
\item[(a)] $\forall s \in [\tau,t_p^{q+1}]$,
$y_l(s) \leq \frac{3}{4} y_l(\tau) + \frac{1}{4} y_{l+1}(\tau)$ and
$y_{l+1}(s) \geq \frac{1}{4} y_{l}(\tau) + \frac{3}{4} y_{l+1}(\tau) $;

\item[(b)] $\exists s \in [\tau,t_p^{q+1}]$, $y_l(s) > \frac{3}{4} y_l(\tau) + \frac{1}{4} y_{l+1}(\tau)$; 

\item[(c)] $\exists s \in [\tau,t_p^{q+1}]$, $y_{l+1}(s) < \frac{1}{4} y_{l}(\tau) + \frac{3}{4} y_{l+1}(\tau)$.
\end{enumerate}
We shall show successively that in all three cases, the expected result holds. In the following, let us denote 
$R =\rrr$. 

{\bf Case (a)}: The assumption of case (a) guarantees that  a minimal distance is preserved between $y_{l+1}(s)$ and $y_l(s)$ for all $s\in [\tau,t_p^{q+1}]$:
\begin{equation}\label{eq:sizeofgap}
y_{l+1}(s)-y_l(s) \geq \frac{1}{2}\left(y_{l+1}(\tau)-y_l(\tau)\right).
\end{equation}
According to Lemma~\ref{l:r-bound-b}, the assumptions of Lemma~\ref{l:adaptation_hendrickx} are satisfied with $y_i:=y_i(s)$ and $b_{ji}:=b_{ji}(s)$ for any $s$ in $[\tau,t_p^{q+1}]\subseteq [0,t_{p+1}]$. Thus, Lemma~\ref{l:adaptation_hendrickx} yields
\begin{equation}\label{eq:hendrickx}
\ssum_{i=1}^l R^{-i} \left( \ssum_{j=1}^n b_{ij}(s) (y_j(s) - y_i(s)) \right) \geq (y_{l+1}(s)-y_l(s))  R^{-l} \ssum_{i=l+1}^n \ssum_{j=1}^{l} b_{ji}(s).
\end{equation}
Equations  (\ref{eq:hendrickx}) and (\ref{eq:sizeofgap}) then give us for all $s\in [\tau,t_p^{q+1}]$:
\begin{equation*}
\ssum_{i=1}^l R^{-i} \left( \ssum_{j=1}^n b_{ij}(s) (y_j(s) - y_i(s)) \right) \geq 
\frac{1}{2}\left(y_{l+1}(\tau)-y_l(\tau)\right) R^{-l} \ssum_{i=l+1}^n \ssum_{j=1}^{l} b_{ji}(s).
\end{equation*}
Then, one can integrate over time interval $[\tau,t_p^{q+1}]$ both sides of the previous inequality to obtain
\begin{equation*}
\ssum_{i=1}^l R^{-i} (y_i(t_p^{q+1})-y_i(\tau)) \geq \frac{1}{2}\left(y_{l+1}(\tau)-y_l(\tau)\right)  R^{-l} \ssum_{i=l+1}^n \ssum_{j=1}^{l} \int_\tau^{t_p^{q+1}} b_{ji}(s) ds.
\end{equation*}
Using the assumption of case (a), Lemma~\ref{l:subbij}, with $c:=\frac{3}{4} y_l(\tau) + \frac{1}{4} y_{l+1}(\tau)$ and $C:=\frac{1}{4} y_l(\tau) + \frac{3}{4} y_{l+1}(\tau)$, states that sets $\sigma_s(\{1,\ldots,l\})$ and $\sigma_s(\{l+1,\ldots,n\})$ remain constant over $[\tau,t_p^{q+1}]$. Therefore, 
$$
\ssum_{i=l+1}^n \ssum_{j=1}^{l} \int_\tau^{t_p^{q+1}} b_{ji}(s) ds
= \ssum_{i=l+1}^n \ssum_{j=1}^{l} \int_\tau^{t_p^{q+1}} a_{\sigma_\tau(j)\sigma_\tau(i)}(s) ds
= \ssum_{j' \in S} \ssum_{i' \in \NN \setminus S} \int_\tau^{t_p^{q+1}} a_{j'i'}(s) ds
$$
where $S = \sigma_\tau(\{1,\ldots,l\})$ (since $\sigma_\tau$ is a bijection, the change of index is possible).
This is where we make use of the persistent connectivity assumption. Since interval $[\tau,t^{q+1}]$ includes $[t^q,t^{q+1}]$ and the terms inside the integrals are positive, equation (\ref{eq:rescaling}) gives 
$$\ssum_{j' \in S} \ssum_{i' \in \NN \setminus S} \int_\tau^{t_p^{q+1}} a_{j'i'}(s) ds \geq 1.$$
So,
\begin{equation*}
\ssum_{i=1}^l R^{-i} (y_i(t_p^{q+1})-y_i(\tau)) \geq \frac{1}{2}\left(y_{l+1}(\tau)-y_l(\tau)\right)  R^{-l}.
\end{equation*}
Denote $m\in\{1,\ldots,l\}$ the maximal argument of the sum on the left hand side of the previous inequality. 
Let $\tau'=t_p^{q+1}$, then
\begin{equation*}
l R^{-m} (y_{m}(\tau')-y_{m}(\tau)) \geq  \frac{1}{2}\left(y_{l+1}(\tau)-y_l(\tau)\right)  R^{-l}
\end{equation*}
which using $l \le n$ gives us with $R \ge 1$
\begin{equation}
\label{eq:case1}
y_{m}(\tau')-y_{m}(\tau) \geq \frac{1}{2l}\left(y_{l+1}(\tau)-y_l(\tau)\right)  R^{m-l}
\geq \frac{1}{4n}\left(y_{l+1}(\tau)-y_l(\tau)\right)  R^{m-(l+1)}.
\end{equation}

{\bf Case (b)}: Let $m = l$ and $\tau' = s$. Then using the assumption of case (b), 
and $m = l$ with $R \ge 1$,
we obtain
\begin{equation}
\label{eq:case2}
y_{m}(\tau')-y_{m}(\tau) \geq \frac{1}{4}\left(y_{l+1}(\tau)-y_l(\tau)\right) 
\geq \frac{1}{4n}\left(y_{l+1}(\tau)-y_l(\tau)\right)  R^{m-(l+1)}.
\end{equation}

{\bf Case (c)}:
According to Lemma~\ref{l:r-bound-b}, the assumptions of Lemma~\ref{l:adaptation_hendrickx} are satisfied with $l:=l+1$, $y_i:=y_i(t)$ and $b_{ji}:=b_{ji}(t)$ for any $t\in [\tau,t_p^{q+1}]\subseteq [0,t_{p+1}]$. Thus, Lemma~\ref{l:adaptation_hendrickx} yields
for all $t\in [\tau,t_p^{q+1}]$
\begin{equation*}
\ssum_{i=1}^{l+1} R^{-i} \left( \ssum_{j=1}^n b_{ij}(t) (y_j(t) - y_i(t)) \right) \geq 0.
\end{equation*}
Integrating this inequality over interval $[\tau,s]\subseteq [\tau,t_p^{q+1}]$ gives
\begin{equation*}
\ssum_{i=1}^{l+1} R^{-i} \left( y_i(s) - y_i(\tau) \right)=
\ssum_{i=1}^{l} R^{-i} \left( y_i(s) - y_i(\tau) \right)
+R^{-(l+1)} \left( y_{l+1}(s) - y_{l+1}(\tau) \right)
 \geq 0.
\end{equation*}
The assumption of case (c) then gives
$$
\ssum_{i=1}^{l} R^{-i} \left( y_i(s) - y_i(\tau) \right)
\ge  R^{-(l+1)} \left( y_{l+1}(\tau) - y_{l+1}(s) \right)\ge 
R^{-(l+1)} \frac{1}{4}\left(y_{l+1}(\tau)-y_{l}(\tau)\right).
$$
Denote $m\in\{1,\ldots,l\}$ the maximal argument of the sum on the left hand side of the previous inequality, 
let $\tau'=s$. Then
\begin{equation*}
l R^{-m}(y_{m}(\tau')-y_{m}(\tau)) \geq R^{-(l+1)} \frac{1}{4}\left(y_{l+1}(\tau)-y_{l}(\tau)\right).
\end{equation*}
which using $l \le n$ rewrites to
\begin{equation}
\label{eq:case3}
y_{m}(\tau')-y_{m}(\tau) \geq R^{m-(l+1)} \frac{1}{4l}\left(y_{l+1}(\tau)-y_{l}(\tau)\right) \ge  
R^{m-(l+1)} \frac{1}{4n}\left(y_{l+1}(\tau)-y_{l}(\tau)\right).
\end{equation}
\end{proof}

\begin{lemma}\label{l:diameter-ind} 
Let $q \in \{0,\ldots,\nn-1\}$, $\tau \in [t_p,t_p^q]$ and $l \in \{1,\ldots,n-1\}$ such that
$y_{l}(\tau)<y_{l+1}(\tau)$.
Then, there exists $\tau'\in [t_p,t_p^{q+1}]$ such that
one of the following assertions holds:
\begin{equation*}
y_{1}(\tau')-y_1(\tau) \geq \frac{\rrr^{-l}}{8n}\left( y_{l+1}(\tau)-y_{l}(\tau) \right)
\end{equation*}
or
\begin{equation*}
\exists l' \in \{1,\ldots,l-1\},\;
y_{l'+1}(\tau')-y_{l'}(\tau') \geq \frac{\rrr^{l'-l}}{8n^2}\left( y_{l+1}(\tau)-y_{l}(\tau) \right).
\end{equation*}
\end{lemma}

\begin{proof} From Lemma~\ref{l:relative-increase}, we obtain that there exists $m \in \{1,\ldots,l\}$  and $\tau'\in [t_p,t_p^{q+1}]$ such that
\begin{equation*}
y_{m}(\tau') - y_{m}(\tau) \geq \frac{\rrr^{m-(l+1)}}{4n}\left( y_{l+1}(\tau)-y_{l}(\tau) \right),
\end{equation*}

Let us first handle the case where $m=1$ (note that this is always the case if $l=1$). In that case, the previous inequality gives
\begin{equation*}
y_{1}(\tau') - y_{1}(\tau) \geq \frac{\rrr^{1-(l+1)}}{4n}\left( y_{l+1}(\tau)-y_{l}(\tau) \right)
\geq \frac{\rrr^{-l}}{8n}\left( y_{l+1}(\tau)-y_{l}(\tau) \right)
\end{equation*}
Thus the first assertion holds.

If $m>1$, using the fact that $y_m(\tau)\ge y_1(\tau)$, we have
$$
y_m(\tau')-y_1(\tau')+y_1(\tau')-y_1(\tau)\ge y_m(\tau')- y_{m}(\tau)
 \ge  \frac{\rrr^{m-(l+1)}}{4n}\left( y_{l+1}(\tau)-y_{l}(\tau) \right)
$$
Thus, either 
$$
y_{1}(\tau') - y_{1}(\tau) \geq \frac{\rrr^{m-(l+1)}}{8n}\left( y_{l+1}(\tau)-y_{l}(\tau) \right)
$$
or
$$
y_m(\tau')-y_1(\tau') \ge \frac{\rrr^{m-(l+1)}}{8n}\left( y_{l+1}(\tau)-y_{l}(\tau) \right).
$$
In the first case, $m> 1$ gives that the first assertion holds. In the second case, let us remark that
$$
y_m(\tau')-y_1(\tau') =\sum_{i=1}^{m-1} \left(y_{i+1}(\tau')-y_i(\tau')\right). 
$$
Let $l'\in \{1,\dots,m-1\}\subseteq \{1,\dots,l-1\}$ be the maximal argument of this sum. Then, 
$$
(m-1)\left(y_{l'+1}(\tau')-y_{l'}(\tau')\right)\ge \frac{\rrr^{m-(l+1)}}{8n}\left( y_{l+1}(\tau)-y_{l}(\tau) \right)
$$
which gives using the fact that $l'\le m-1 \le n$ and $\rrr \ge 1$ 
$$
y_{l'+1}(\tau')-y_{l'}(\tau') \ge \frac{\rrr^{m-(l+1)}}{8n(m-1)}\left( y_{l+1}(\tau)-y_{l}(\tau) \right)
\ge \frac{\rrr^{l'-l}}{8n^2}\left(y_{l+1}(\tau)-y_{l}(\tau) \right).
$$
Thus the second assertion holds.
\end{proof}

\begin{proof}[Proof of Proposition~\ref{prop:small-move}]
Let us recall that we must show that 
\begin{equation}
\label{eq:goal}
\Delta_\NN(t_{p+1}) \leq \left(1 - \frac{\rrr^{-\nn}}{(8n^2)\nn} \right)
\Delta_\NN(t_p).
\end{equation}
If $\Delta_\NN(t_p)=0$, since $\Delta_\NN$ is non-increasing we obtain $\Delta_\NN(t_{p+1})=0$
and (\ref{eq:goal}) clearly holds.

Let us assume that $\Delta_\NN(t_p)>0$. Let us remark that
$$
\Delta_\NN(t_p) =\sum_{i=1}^{n-1} \left(y_{i+1}(t_p)-y_i(t_p)\right). 
$$
Let $l^0\in \{1,\dots,n-1\}$ be the maximal argument of this sum. Then, 
\begin{equation}
\label{eq:init}
y_{l^0+1}(t_p)-y_{l^0}(t_p) \ge \frac{1}{n-1}\Delta_\NN(t_p) \ge  \frac{1}{n}\Delta_\NN(t_p) .
\end{equation}
Let us first assume that $l^0 \le \nn$ (the case $l^0>\nn$ will be discussed later).
We now proceed by induction on Lemma~\ref{l:diameter-ind}. There exists a finite decreasing sequence
$l^0>l^1>\dots>l^{\bar q}$ such that $l^{\bar q}=1$ (which implies $\bar q \le \nn-1$) and a finite sequence 
$\tau^0,\tau^1,\dots,\tau^{\bar q},\tau^{\bar q+1}$ such that $\tau^0=t_p$, $\tau^q\in [t_p,t_p^q]\subseteq [t_p,t_{p+1}]$ for all 
$q\in \{1,\dots,\bar q+1\}$ and
\begin{equation*}
y_{l^{q}+1}(\tau^{q})-y_{l^{q}}(\tau^{q}) \geq \frac{\rrr^{l^{q}-l^{q-1}}}{8n^2}\left( y_{l^{q-1}+1}(\tau^{q-1})-y_{l^{q-1}}(\tau^{q-1}) \right), \; \forall q\in \{1,\dots,\bar q\}
\end{equation*}
and
\begin{equation*}
y_{1}(\tau^{\bar q+1})-y_{1}(\tau^{\bar q}) \geq \frac{\rrr^{-l^{{\bar q}}}}{8n}\left( y_{l^{{\bar q}}+1}(\tau^{{\bar q}})-y_{l^{{\bar q}}}(\tau^{{\bar q}}) \right).
\end{equation*}
The previous equations together with (\ref{eq:init}) give
\begin{eqnarray*}
y_{1}(\tau^{\bar q+1})-y_{1}(\tau^{\bar q}) & \geq &\frac{\rrr^{-l^{{\bar q}}}}{8n}
\left(\prod_{q=1}^{\bar q} \frac{\rrr^{l^{q}-l^{q-1}}}{8n^2}\right)
\left( y_{l^{0}+1}(\tau^{0})-y_{l^{0}}(\tau^{0}) \right)\\
& \ge &  \frac{\rrr^{-l^{0}}}{(8n^2)^{\bar q +1}} \Delta_\NN(t_p).
\end{eqnarray*}
Using the fact that $l^0\le \nn$, $\bar q \le \nn-1$ and $\rrr \ge 1$, we can write
\begin{equation}
\label{eq:pr1}
y_{1}(\tau^{\bar q+1})-y_{1}(\tau^{\bar q}) \ge \frac{\rrr^{-\nn}}{(8n^2)^{\nn}} \Delta_\NN(t_p).
\end{equation}
$\tau^{\bar q}$ and $\tau^{\bar q+1}$ both belong to $[t_p,t_{p+1}]$ then since $y_1$ is non-decreasing, we have
$y_1(t_p)\le y_1(\tau^{\bar q})$ and $y_1(\tau^{\bar q+1}) \le y_1(t_{p+1})$. Therefore,
\begin{equation}
\label{eq:pr2}
y_1(t_{p+1})-y_1(t_p) \ge y_{1}(\tau^{\bar q+1})-y_{1}(\tau^{\bar q}).
\end{equation}
Also $y_n$ is non-increasing. Therefore, $y_n(t_{p+1}) \le y_n(t_p)$ and
\begin{eqnarray*}
\Delta_\NN(t_{p+1}) &=& y_n(t_{p+1})-y_1(t_{p+1})  \\
&\le& y_n(t_p) -y_1(t_p) +y_1(t_p) -y_1(t_{p+1}) = 
\Delta_\NN(t_{p}) +y_1(t_p) -y_1(t_{p+1}). 
\end{eqnarray*}
This inequality with (\ref{eq:pr1}) and (\ref{eq:pr2}) gives (\ref{eq:goal}).

If $l^0 \ge \nn+1$, then by adapting Lemmas~\ref{l:relative-increase} and~\ref{l:diameter-ind}, we can show that there exists a finite increasing sequence
$l^0<l^1<\dots<l^{\bar q}$ such that $l^{\bar q}=n$ (which implies $\bar q \le \nn-1$) and a finite sequence 
$\tau^0,\tau^1,\dots,\tau^{\bar q},\tau^{\bar q+1}$ such that $\tau^0=t_p$, $\tau^q\in [t_p,t_p^q]\subseteq [t_p,t_{p+1}]$ for all 
$q\in \{1,\dots,\bar q+1\}$ and
$$
y_{n}(\tau^{\bar q+1})-y_{n}(\tau^{\bar q}) \le - \frac{\rrr^{-\nn}}{(8n^2)^{\nn}} \Delta_\NN(t_p).
$$
Similar to the previous case, this equation would also lead to (\ref{eq:goal}).
This ends the proof of Proposition~\ref{prop:small-move}.
\end{proof}

\section{Conclusion}

In this paper, we have developed new results on consensus for continuous multi-agent systems. 
We have introduced the assumptions of persistent connectivity (Assumption~\ref{as:persistent-conn}) and slow divergence of reciprocal interaction weights (Assumption~\ref{as:weak-sym}). Our assumption on the reciprocal weight ratio generalizes the cut-balanced weights assumption proposed recently in~\cite{Hendrickx2011}. Unlike the cut-balanced weights assumption, it enables for indefinitely divergent reciprocal weights which, to the best of the authors' knowledge, is completely new. Moreover, our proof allows us to give an estimate of the rate of convergence to the consensus which is not available in~\cite{Hendrickx2011}.
We have also provided  examples illustrating why we believe that both of our assumptions are tight.
In the future, we would like to relax the assumption on reciprocal weights so that reciprocity does not need to instantaneous (i.e. if a group $S$ of agents receives some influence from the rest of the agents, these will be influenced back by agents of $S$ in the future).

\bibliographystyle{alpha}
\bibliography{references}

\end{document}